\documentclass{amsart}

\usepackage{hyperref}
\usepackage{epsfig}
\usepackage{appendix}
\usepackage{amsfonts,amssymb,amsmath,amsthm,amscd}
\usepackage{latexsym}
\usepackage{mathtools}
\usepackage{graphicx}
\usepackage[all, cmtip]{xy}
\usepackage{xcolor}
\usepackage{tikz-cd}
\usepackage{comment}
\usepackage[normalem]{ulem}

\newcommand{\A}{{\mathbb{A}}}

\newcommand{\C}{{\mathbb{C}}}

\newcommand{\G}{{\mathbb{G}}}

\newcommand{\I}{{\mathbb{I}}}

\newcommand{\Q}{{\mathbb{Q}}}
\newcommand{\R}{{\mathbb{R}}}

\newcommand{\Z}{{\mathbb{Z}}}

\newcommand{\Ccal}{{\mathcal{C}}}

\newcommand{\Ocal}{{\mathcal{O}}}

\newcommand{\red}{\textup{red}}

\DeclareMathOperator{\Cl}{Cl}

\DeclareMathOperator{\GL}{GL}

\DeclareMathOperator{\Hom}{Hom}

\DeclareMathOperator{\sheafhom}{\mathcal{H}\kern -.5pt \emph{om}}
\DeclareMathOperator{\Img}{Im}

\DeclareMathOperator{\Mat}{Mat}

\DeclareMathOperator{\pr}{pr}

\DeclareMathOperator{\PSL}{PSL}

\DeclareMathOperator{\SL}{SL}
\DeclareMathOperator{\SO}{SO}

\DeclareMathOperator{\Span}{Span}

\DeclareMathOperator{\tr}{tr}

\newcommand{\git}{\mathbin{
  \mathchoice{/\mkern-6mu/}
    {/\mkern-6mu/}
    {/\mkern-5mu/}
    {/\mkern-5mu/}}}

\newtheorem{innercustomthm}{Theorem}
\newenvironment{thm}[1]
  {\renewcommand\theinnercustomthm{#1}\innercustomthm}
  {\endinnercustomthm}

\theoremstyle{plain}
		\newtheorem{theorem}{Theorem}[section]
		\newtheorem{lemma}[theorem]{Lemma}
		\newtheorem{corollary}[theorem]{Corollary}
		\newtheorem{proposition}[theorem]{Proposition}

		\newtheorem*{corollary*}{Corollary}

		\newtheorem{fact}[theorem]{Fact}
\theoremstyle{definition}
		\newtheorem{definition}[theorem]{Definition}
		\newtheorem{def-prop}[theorem]{Definition-proposition}

\theoremstyle{remark}
		\newtheorem*{remark}{Remark}
		
\DeclareMathOperator{\Ort}{O}
\DeclareMathOperator{\Pin}{Pin}
\DeclareMathOperator{\Spin}{Spin}
				
\begin{document}
\title{Short words of infinite order}

\author[Junho Peter Whang]{Junho Peter Whang}

\date{\today}

\maketitle

\begin{abstract}
Given an infinite linear group with a finite set of generators, we show that the shortest word length of an element of infinite order has an upper bound that depends only on the number of generators and the degree. This provides a quantification of the Burnside problem for linear groups. In degree two, an explicit bound is computed using an exceptional connection to reflection groups.
\end{abstract}

\setcounter{tocdepth}{1}
\tableofcontents

\section{Introduction}
\label{sect:1}
\subsection{Main results}
The first main result of this paper is the following.

\begin{theorem}
\label{word}
    Given integers $r,n\geq1$, there exists an integer $\ell=\ell(r,n)\geq0$ such that, for any finite set $S$ of $r$ matrices in $\GL_n(\C)$ generating a group $G$ of infinite order, there is an element of infinite order in $G$ with word length $\leq\ell$ in $S$.
\end{theorem}

Here, the word length of an element $g\in G$ in $S$ is the minimum integer $k\geq0$ such that $g=s_1\cdots s_k$ for some $s_1,\dots,s_k\in S\cup S^{-1}$. The nontrivial content of Theorem \ref{word} lies in the case where $S$ consists only of torsion elements. We point out a loose resemblance to the systolic inequality in Riemannian geometry: by work of Gromov \cite{gromov}, the length of the shortest noncontractible loop in an essential closed Riemannain manifold has an upper bound that depends only on the dimension and the volume of the manifold. Theorem \ref{word} is inspired by the classical Burnside problem for linear groups. Recall that, by Schur's theorem \cite{schur}, a complex linear representation of a finitely generated group $F$ has finite image if and only if the image of each element of $F$ is a torsion element. A natural question is whether torsionness of image for only a small part of $F$ is enough to ensure that a given representation is finite. This leads us to the following definition.

\begin{definition}
\label{burnside}
    Let $F$ be a group, and let $\Ccal$ be a class of group homomorphisms with domain $F$. Let $L\subseteq F$ be a subset. We say that $L$ is a \emph{Burnside set} for $\Ccal$ if the following are equivalent for every $\rho\in\Ccal$:
    \begin{enumerate}
        \item $\rho(\gamma)$ is torsion for all $\gamma\in L$.
        \item $\rho$ has finite image.
    \end{enumerate}
\end{definition}

A motivating example is where $F=\pi_1(\Sigma)$ is the fundamental group of a surface $\Sigma$ of genus $g$ with $n$ punctures, and $L$ is the set of its simple loops. If $3g+n-3>0$, the simple loops form a sparse yet infinite collection of elements in the fundamental group. Patel-Shankar-Whang \cite{psw} showed that, for $g>0$, the simple loops form a Burnside set for the class of semisimple degree $2$ representations of $F$ over $\C$. This was used in the proof of the $p$-curvature conjecture in rank $2$ for generic curves \emph{loc.,cit.}, and also in the classification of $\SL_2(\C)$-local systems with finite mapping class group orbits by Biswas-Gupta-Mj-Whang \cite{bgmw}. (On the other hand, Koberda-Santharoubane \cite{ks} showed that the set of simple loops cannot be a Burnside set for \emph{all} semisimple representations of a given surface group over $\C$.) The following is a reformulation of Theorem \ref{word}.

\begin{theorem}
\label{mainthm}
For any finitely generated group $F$ and integer $n\geq1$, there exists a finite Burnside set for $\Hom(F,\GL_n(\C))$.
\end{theorem}

We prove Theorem \ref{mainthm} by combining Schur's theorem with Laurent's solution \cite{laurent} of Lang's $\G_m$ conjecture (See also \cite{s,sa})
and Procesi's work \cite{procesi} on invariant theory of matrices. (We mention in passing that Lang's $\G_m$ conjecture was also used in \cite{crrz} to show that anisotropic linear groups that are boundedly generated must be virtually solvable.) Our proof is nonconstructive, and it raises the interesting problem of explicitly constructing finite Burnside sets for $\Hom(F,\GL_n(\C))$. For $n=1$ or more generally for the class of abelian representations of $F$, it is trivial that any finite generating set of $F$ provides a Burnside set. Our second result is a solution to this problem in the simplest nontrivial case, where $n=2$.

\begin{theorem}
\label{mainthm2}
Let $F$ be a group generated by a finite set $S$. The set of elements of word length $\leq3|S|$ in $S$ is a Burnside set for $\Hom(F,\GL_2(\C))$.
\end{theorem}

In fact, we produce an explicit finite set of words in $S$ that forms a Burnside set for $\Hom(F,\GL_2(\C))$. To prove Theorem \ref{mainthm2}, we first introduce and study reflective representations, which are representations $F_r\to\GL_n(\C)$ that send the standard free generators of $F_r$ to orthogonal reflections (Definition \ref{reflective}). Theorem \ref{mainthm2} can be deduced from Theorem \ref{refthm} below, using an exceptional correspondance established by Fan-Whang \cite{fw} between free group representations of degree $2$ and reflective representations of degree $4$.

\begin{theorem}
\label{refthm}
    Let $F_r=\langle\gamma_1,\dots,\gamma_r\rangle$ be a free group of rank $r$. Then $$L_{r}=\{\gamma_{i_1}\dots \gamma_{i_u}:1\leq i_1<\dots <i_u\leq r\text{ and }1\leq u\leq r\}$$
    is a Burnside set for the class of all semisimple reflective representations of $F_r$.
\end{theorem}

Note that the set $L_r$ defined above is independent of the degrees of the reflective representations. Theorem \ref{refthm} is closely related to the fact that a Coxeter group is finite if and only if it admits a positive definite cosine matrix. The tools involved in our proof are classical: aside from Schur's theorem, we make use of the Coxeter identity, Kronecker's theorem on roots of unity, Sylvester's criterion on definiteness of matrices, and a Galois trick employed earlier in \cite{psw}.

Theorem \ref{word} is related to the following result of Breuillard-Gelander \cite{bg} on the Tits alternative: if $G$ is a finitely generated linear group that is not virtually solvable, then there is a constant $m=m(G)$ such that, for any finite generating set $S$ of $G$, there exist two elements $a,b\in G$ each with word length $\leq \ell$ in $S$ that are independent, i.e.,~generate a free nonabelian subgroup of rank $2$ in $G$. Note that our Theorem \ref{word} yields a weaker conclusion (existence of a single ``independent'' short word) from a weaker hypothesis ($G$ need only be finitely generated and infinite). In addition, if the degree is fixed, then our upper bound on word length depends on the order of the chosen generating set but is otherwise is independent of the group, while the bound in \cite{bg} is depends on the group but not on the generating set.

\subsection{Applications to surface group representations} A fruitful observation is that, for a natural presentation of the free group $F_r$ as the fundamental group of $\C-\{p_1,\dots,p_r\}$ for some distinct marked points $p_1,\dots,p_r$ in the complex plane, the elements of $L_r$ in Theorem \ref{refthm} correspond to simple loops. Thus, there exists a finite Burnside set consisting of simple loops for the class of semisimple reflective representations of $\pi_1(\C-\{p_1,\dots,p_r\})$. By passing to a double cover and utilizing the exceptional correspondance of \cite{fw}, we can show that if $\Sigma$ is a surface of positive genus with at most two punctures, then there is an explicitly determined Burnside set consisting of \emph{finitely many} simple loops for the set of all semisimple degree $2$ representations of $\pi_1(\Sigma)$ over $\C$ (Corollary \ref{surfcor}); this provides an effective strengthening of \cite[Theorem 1.2]{psw} for $\Sigma$. In addition, Theorem \ref{refthm} yields a new case of the $p$-curvature conjecture, when combined with Shankar's theorem \cite{shankar} that vector bundles with flat connection on a generic curve with almost all $p$-curvatures vanishing must have finite monodromy along simple loops.

\begin{theorem}\label{theorem2}
Let $C$ be the complement of a finite generic set of points in the affine line $\A^1$. The $p$-curvature conjecture holds for vector bundles with flat connection on $C$ whose monodromy representations are reflective.
\end{theorem}

The $p$-curvature conjecture of Grothendieck-Katz is a type of local-to-global principle for differential equations. It states that a system of linear differential equations on an algebraic variety over $\C$ admits a full set of algebraic solutions over $\C$ if and only if it does so modulo $p$ (or, equivalently, has vanishing $p$-curvature) for all but finitely many prime numbers $p$; see \cite{katz} for a reference.
 
Theorem \ref{theorem2} applies in particular to the following Fuchsian system of differential equations, encountered in two-dimensional topological field theory, whose monodromy representations define the so-called monodromy groups of Frobenius manifolds (See \cite[Section 5]{dubrovin}, in particular Equation (5.31b) and surrounding paragraphs). Fix a skew-symmetric matrix $A\in\mathfrak{so}(r)$, and let $p_1,\dots,p_r\in\C$ be distinct points. For $1\leq k\leq r$, let $E_k$ be the $r\times r$ matrix with unique nonzero entry $1$ in the $(k,k)$th place, and let $\mathbb I$ be the identity matrix. Consider the following Fuchsian system for the vector valued function $Y(z)$:
$$\frac{d}{dz}Y=\sum_{k=1}^r\frac{E_k(\frac{1}{2}\mathbb I-A)}{z-p_i}Y.$$
Assuming generic choices of the points $p_i$, Theorem \ref{theorem2} shows that the $p$-curvature conjecture holds for the system above.

\subsection{Organization of the paper} This paper is organized as follows. In Section \ref{sect:2}, we prove Theorem \ref{mainthm}. In Section \ref{sect:3}, we introduce the notion of Stokes matrices and their associated representations, proving the equivalent of Theorem \ref{refthm}. In Section \ref{sect:3}, we prove Theorems \ref{refthm} and \ref{mainthm2} using Stokes matrices and the exceptional correspondance established in \cite{fw}.

\subsection{Acknowledgments}
I thank Peter Sarnak and Alexander Lubotzky for enlightening discussions and comments.
This work was supported by the Samsung Science and Technology Foundation under Project Number SSTF-BA2201-03.




\section{Finite Burnside sets}\label{sect:2}

\subsection{Burnside-Schur theory}\label{sect:2.1}
Our goal in Section \ref{sect:2.1} is to show (Proposition \ref{burn}) that the finiteness of a representation $\rho:F_r\to\GL_n(\C)$ is essentially equivalent to the quasiunipotency of the elements in the image of $\rho$, modulo considerations of global semisimplicity. We first recall the statements of some relevant results, including Schur's theorem.

\begin{theorem}
    [Schur \cite{schur}]\label{schur}
    Let $\Gamma\leq\GL_n(\C)$ be a finitely generated group. If $\gamma$ has finite order for every $\gamma\in\Gamma$, then $\Gamma$ is finite.
\end{theorem}

\begin{theorem}
    [Jordan's theorem, see \cite{cr}]\label{js}
    There exists an effective constant $N(n)$ such that, for any finite group $G\leq\GL_n(\C)$, there exists a normal abelian subgroup $H$ of index $\leq N(n)$ in $G$.
\end{theorem}

\begin{theorem}
    [Bass \cite{bass}]\label{bass}
    Let $G$ be a subgroup of $\GL_n(\C)$ that acts irreducibly on $\C^n$. If the set of traces of all elements in $G$ is a bounded subset of $\C$, then $G$ is bounded, i.e.,~conjugate to a subgroup of $U(n)$.
\end{theorem}

Recall that an element $g\in\GL_n(\C)$ is said to be \emph{quasiunipotent} if every eigenvalue of $g$ is a root of unity. By combining the above ingredients, we deduce the following result which, in essence, allows us to reduce the problem of finiteness of a representation $\rho:F_r\to\GL_n(\C)$ to an infinite system of trigonometric Diophantine equations in infinitely many unknowns (eigenvalues of $\rho(\gamma)$ as $\gamma$ ranges over all elements of $F_r$). This will be used in our proof of Theorem \ref{mainthm}.

\begin{proposition}
\label{burn}
    There is an effectively determined finite set $U_{r,n}\subseteq F_r$ such that a representation $\rho:F_r\to\GL_n(\C)$ has finite image if and only if
    \begin{enumerate}
        \item $\rho(\gamma)$ has finite order for every $\gamma\in U_{r,n}$, and
        \item $\rho(\gamma)$ is quasiunipotent for every $\gamma\in F_r$.
    \end{enumerate}
\end{proposition}

\begin{proof}
    By Theorem \ref{js}, there is an effective constant $N(n)\geq1$ such that any finite subgroup of $\GL_n(\C)$ has a normal abelian subgroup of index $\leq N(n)$. Let us choose a finite subset $T_{r,n}$ of $F_r$ such that, for every group homomorphism $\varphi:F_r\to G$ where $G$ is a finite group of order at most $N(n)$, there is a subset of $T_{r,n}$ that generates $\ker(\varphi)$. We claim that the set
    $$U_{r,n}=T_{r,n}\cup \{[a,b]:a,b\in T_{r,n}\}.$$
    satisfies the conclusions of the proposition. Indeed, let $\rho:F_r\to\GL_n(\C)$ be a representation satisfying conditions (1) and (2). If $\rho$ is irreducible, then $\rho$ is unitarizable by Theorem \ref{bass} and condition (2). It follows that every $\rho(\gamma)$ has finite order for $\gamma\in F_r$, showing that $\rho$ has finite image by Theorem \ref{schur}.
    
    It remains to treat the case where $\rho$ is reducible. The semisimplification $\rho^s$ of $\rho$ has finite image, by arguing as above. Now, by Theorem \ref{js} we can choose a subset $S_\rho\subseteq T_{r,n}$ such that its image under $\rho^s$ generates a normal abelian subgroup of index $\leq N(n)$ in $\rho^s(F_r)$. In particular, $S_\rho$ generates a finite index subgroup of $F_r$. We may assume that $\rho(F_r)$ consists of block upper triangular matrices (with each diagonal block corresponding to an irreducible summand of $\rho^s$), and moreover that $\rho(S_\rho)$ belongs to the group of upper triangular matrices. Given any $a,b\in S_\rho$, it follows that the commutator $[\rho(a),\rho(b)]$ is a unipotent matrix and hence has finite order if and only if it is the identity matrix. Our choice of $U_{r,n}$ and the hypothesis on $\rho$ implies therefore that the restriction of $\rho$ to $\langle S_\rho\rangle$ is abelian, and hence $\rho(\langle S_\rho\rangle)$ is finite since $\rho(\gamma)$ has finite order for every $\gamma\in S_\rho\subseteq T_{r,n}\subseteq U_{r,n}$. Since $\langle S_\rho\rangle$ has finite index in $F_r$, we conclude that $\rho$ has finite image.
\end{proof}

\subsection{Finite Burnside sets} We turn to the proof of Theorem \ref{mainthm}. In addition to Proposition \ref{burn}, we shall need Laurent's solution of Lang's $\G_m$ conjecture and Procesi's theory on invariants of $n\times n$ matrices.

Let $\G_m=\C^\times$ denote the multiplicative group. Let $n\geq1$ be an integer. By a torsion point on $\G_m^n=(\C^\times)^n$ we shall mean a point $(\zeta_1,\dots,\zeta_n)\in\G_m^n$ all of whose coordinates are roots of unity. Given a closed subvariety $Y$ of $\G_m^n$, by a torsion point on $Y$ we shall mean a point of $Y$ that is a torsion point of $\G_m^n$.

\begin{theorem}
    [Lang's $\G_m$ conjecture, Laurent \cite{laurent}] \label{lang}
    Let $Y$ be a closed subvariety of $\G_m^n$. Then the Zariski closure of the set of torsion points on $Y$ is a finite union of effectively determined torsion cosets of linear subtori of $\G_m^n$.
\end{theorem}

Let now $r,n\geq1$ be integers. Let $\Mat_n\simeq\A^{n^2}$ be the affine scheme parametrizing $n\times n$ matrices. The group $\GL_n$ acts on $\Mat_n^r$ by simultaneous conjugation. Let $X_i$ denote the $i$th matrix variable, whose $n\times n$ coordinates form elements generating the coordinate ring $\Q[\Mat_n^r]$ of $\Mat_n^r$. It is clear that regular functions of the form $\tr(X_{i_1}\dots X_{i_j})$ are $\GL_n$-invariant. Procesi shows that in fact all $\GL_n$-invariant regular functions on $\Mat_n^r$ are polynomial combinations of such trace functions.

\begin{theorem}
    [Procesi \cite{procesi}] \label{procesi}
    The ring $\Q[\Mat_n^r]^{\GL_n}$ is finitely generated by elements of the form $\tr(X_{i_1}\dots X_{i_j})$ with $j\leq2^n-1$ and $i_k\in\{1,\dots,n\}$ for $k=1,\dots,j$.
\end{theorem}

We now restate and prove Theorem \ref{mainthm}.

\begin{thm}{\ref{mainthm}}
For any finitely generated group $F$ and integer $n\geq1$, there exists a finite Burnside set for $\Hom(F,\GL_n(\C))$.
\end{thm}

\begin{proof}
It is enough to consider the case where $F=F_r$ is a free group of rank $r\geq1$. It is also enough to show that a finite Burnside set exists for $\Hom(F_r,\SL_n(\C))$ for all $r,n\geq1$, since we have an embedding $\GL_n(\C)\to\SL_{n+1}(\C)$ given by
$$g\mapsto\begin{bmatrix}
    g & 0\\ 0 & \det(g)^{-1}
\end{bmatrix}.$$
Let $\gamma_1,\dots,\gamma_r$ be a set of free genrators of $F_r$. For $\ell\in\Z_{\geq0}$, let us write $L(\ell)$ for the subset of $F_r$ consisting of those $\gamma\in F_r$ of word length at most $\ell$ in $S=\{\gamma_1,\dots,\gamma_r\}$.

Let $X=X(F_r,\SL_n)=\SL_n^r\git\GL_n$ be the $\SL_n$-character variety of $F_r$. It is constructed as the invariant theoretic quotient of $\SL_n^r$ by the diagonal conjugation action of $\GL_n$. By Theorem \ref{mainthm} or Hilbert's basis theorem, $X$ is an integral scheme of finite type over $\Q$. For each $w\in F_r$, let us write $s_1(w),\dots,s_n(w)\in\Q[X]$ for the regular functions on $X$ given by the coefficients of the characteristic polynomial of the image of $w$. More precisely,
$$\det(\lambda-\rho(w))=\lambda^n+\sum_{i=1}^n s_i(\rho(w))\lambda^{n-i}.$$
Let us introduce the following notation. For each $\ell\geq0$, let us define the scheme $Y_\ell$ by the fiber product diagram
$$\xymatrix{Y_\ell \ar[r]\ar[d] & X\ar[d]\\(\G_m^n)^{L(\ell)}\ar[r] & (\A^n)^{L(\ell)}}$$
where the right vertical arrow is given by $\rho\mapsto(s_1(\rho(w)),\dots,s_n(\rho(w)))_{w\in L(\ell)}$, and the bottom horizontal arrow is given by sending each point of $(\G_m^n)^{L(\ell)}$ denoted $(e_1(w),\dots,e_n(w))_{w\in L(\ell)}$ to the elementary symmetric polynomials in the entries:
$$\left(\sum_{i=1}^ne_i(w),\cdots,\prod_{i=1}^ne_i(w)\right)_{w\in L(\ell)}.$$
Note that there is an action of $(S_n)^{L(\ell)}$ on $(\G_m^n)^{L(\ell)}$ by obvious permutations, with respect to which the bottom horizontal arrow is equivariant. Similarly, we have $(S_n)^{L(\ell')\setminus L(\ell)}$-invariant projections $$\pi_{\ell',\ell}:(\G_m^n)^{L(\ell')}\to (\G_m^n)^{L(\ell)}$$
for $\ell'\geq\ell$, inducing invariant morphisms $\pi_{\ell',\ell}: Y_{\ell'}\to Y_\ell$.

By Procesi's Theorem (Theorem \ref{procesi}), there is an effectively determined $\ell_0\geq0$ such that $X\to(\A^n)^{L(\ell)}$ is a closed immersion for all $\ell\geq\ell_0$. Thus, for $\ell\geq\ell_0$ the morphism $Y_{\ell}\to(\G_m^n)^{L(\ell)}$ is a closed immersion. We shall choose $\ell_0$ so large that in fact $L(\ell_0)$ contains the set $U_{r,n}$ constructed in Lemma \ref{bass}.
    For $\ell\geq\ell_0$, let $Z_{\ell}$ denote the Zariski closure of the set of torsion points on the closed subscheme $Y_\ell\subseteq(\G_m^n)^{L(\ell)}$. By Theorem \ref{lang}, $Z_{\ell}$ is an effectively determined finite union of torsion cosets of linear subtori in $(\G_m^n)^{L(\ell)}$. For $\ell'\geq\ell\geq\ell_0$, the morphism $\pi: Y_{\ell'}\to Y_\ell$ sends $Z_{\ell'}$ into $Z_{\ell}$. Let us write $$Z_{\ell}'=\pi_{\ell,\ell_0}(Z_{\ell}')\subset Z_{\ell_0}\quad\text{for all}\quad \ell\geq\ell_0.$$ Then we have a descending chain $Z_{\ell_0}=Z_{\ell_0}'\supseteq Z_{\ell_0+1}'\supseteq Z_{\ell_0+2}'\supseteq\dots$ of finite unions of torsion cosets of algebraic subtori in $(\G_m^n)^{L(\ell_0)}$, which must eventually stabilize. Let $$Z=\bigcap_{i=0}^\infty Z_{\ell_0+i}'$$
    and let us choose $t\geq\ell_0$ such that $Z_t'=Z$. If $x$ is a torsion point of $Z$, then $$\pi_{\ell,\ell_0}^{-1}(x)\cap Z_{\ell}$$ contains a torsion point for every $\ell\geq\ell_0$, because the projections $\pi_{\ell,\ell_0}$ are group homomorphisms. But $\pi_{\ell,\ell_0}^{-1}(x)\cap Z_{\ell}$ is finite, by our hypothesis on $\ell_0$; indeed, $(\G_m^n)^{L(\ell)}\to(\A^n)^{L(\ell)}$ has finite fibers and the image of any $y\in \pi_{\ell,\ell_0}^{-1}(x)\cap Z_{\ell}$ in $X\subset(\A^{n})^{L(\ell)}$ is determined by the image of $x$ in $X\subset(\A^{n})^{L(\ell_0)}$. It follows that some (and hence every) point in $\pi_{\ell,\ell_0}^{-1}(x)\cap Z_{\ell}$ is a torsion point in $(\G_m^n)^{L(\ell)}$ for every $\ell\geq\ell_0$. It follows that, if $\rho\in X$ is in the image of $Z$, then $\rho(\gamma)$ is quasiunipotent for all $\gamma\in F_r$. In particular, by Proposition \ref{burn}, for any representation $\rho:F_r\to\SL_n(\C)$, if $\rho(\gamma)$ has finite order for all $\gamma\in L(t)$, then $\rho$ has finite image. This proves the theorem.
\end{proof}

\section{Stokes matrices}\label{sect:3}

\subsection{Stokes matrices}
Fix an integer $r\geq1$. Let us write $[r]=\{1,\dots,r\}$, and endow it with the usual linear ordering. We shall endow any subset $I\subset[r]$ with the induced ordering. Given an $r\times r$ matrix $a=[x_{ij}]$ and $I\subseteq[r]$, we shall write $a_I=[x_{ij}]_{i,j\in I}$ for the corresponding submatrix consisting of the entries whose coordinates belong to $I$. For each $i\in[r]$, let $E_i$ be the $r\times r$ matrix having unique nonzero entry $1$ in the $(i,i)$th place. Thus, $\I=\sum_{i=1}^r E_i$ is the identity matrix. 

\begin{definition}
A \emph{Stokes matrix} of dimension $r$ is an $r\times r$ upper triangular unipotent matrix. Let $V(r)$ denote the space of Stokes matrices of rank $r$.
\end{definition}

Note that, if $s\in V(r)$ is a Stokes matrix, then the submatrix $s_I$ is a Stokes matrix of dimension $|I|$ for any $I\subseteq[r]$. Below, let $F_r$ denote the free group of rank $r$ with free generators $\gamma_1,\dots,\gamma_r$. For $I\subseteq[r]$, we shall write $\gamma_I=\gamma_{i_1}\dots\gamma_{i_{|I|}}$. Given a Stokes matrix, we can define a representation $\rho:F_r\to\GL_r$ as follows.

\begin{definition}[Associated representation]
Let $k$ be a field, and let $s\in V(r)(k)$. 
\begin{enumerate}
    \item Let $\rho_s:F_r\to\GL_r(k)$ be the representation such that
$$\rho_s(\gamma_i)=\mathbb I-E_i(s+s^T)$$ for $i=1,\dots,r$.
    \item Let $U_0^s$ denote the kernel of $s+s^T$ on $U=k^r$, and let $W^s=U/U_0^s$.
    \item Write $\tilde\rho_s$ for the representation $F_r\to\GL(W^s)$ induced by $\rho_s$.
\end{enumerate}
\end{definition}

\begin{lemma}
\label{genc}
Let $k$ be a field. Given $s\in V(r)(k)$, we have the following.
\begin{enumerate}
    \item (Coxeter identity) We have $\rho(\gamma_{[r]})=-s^{-1}s^T$.
    \item The symmetric pairing $s+s^T$ on $k^r$ is preserved by $\rho_s$.    
    \item For any $I\subseteq[r]$, we have $\det(\lambda-\rho_s(\gamma_I))=(\lambda-1)^{r-|I|}\det(\lambda+s_I^{-1}s_I^T)$.
    \item $U_0^s$ is an invariant subspace of $U$ on which $F_r$ acts trivially via $\rho_s$.
    \item $\tilde\rho_s$ preserves the nondegenerate symmetric pairing on $W^s$ induced by $s+s^T$.
\end{enumerate}
\end{lemma}

\begin{proof}
We note that, for each $k$, the product $s(1-E_1(s+s^T))\cdots (1-E_k(s+s^T))$ has the same first $k$ rows as $-s^T$ and the same last $n-k$ rows as $s$. In particular, we have $s\rho(\gamma_{[r]})=-s^T$, proving (1). To prove (2), it suffices to show hat $$\rho_s(\gamma_k)^T(s+s^T)\rho_s(\gamma_k)=s+s^T$$ for each $i=1,\dots,r$. By definition,
\begin{align*}
\rho_s(\gamma_k)^T(s+s^T)\rho_s(\gamma_k)&=(1-E_k(s+s^T))^T(s+s^T)(1-E_k(s+s^T))\\
&=(1-(s+s^T)E_k)(s+s^T)(1-E_k(s+s^T))=s+s^T
\end{align*}
noting that $E_k(s+s^T)E_k=2E_k$. This gives (2). To prove (3), let $v_1,\dots,v_r$ denote the standard basis vectors of $U=k^r$. Let $I=\{i_1<\dots<i_{|I|}\}$. Note that $\rho_s(\gamma_i)$ for each $i\in I$ preserves $U_I:=\Span(v_i:i\in I)$, and acts trivially on the quotient $U/U_I$. Since the action of $\rho_s(\gamma_I)$ on $U_I$ is given by $-s_I^{-1}s_I^{T}$ by (1), we obtain (3). Parts (4) and (5) are clear.
\end{proof}

\begin{definition}
    Let $K$ be a field complete with respect to a nontrivial absolute value $|\cdot|$. We say that a subgroup $G\leq\GL_n(K)$ is \emph{bounded} if it lies in a compact subgroup of $\GL_n(K)$, with respect to the topology induced by $|\cdot|$.
\end{definition}

\begin{proposition}
\label{elliptic}
Let $K$ be field complete with respect to a nontrivial absolute value $|\cdot|$. Let $s\in V(r)(K)$ be a Stokes matrix. If every eigenvalue of $\rho_s(\gamma_I)$ has absolute value $1$ for every $I\subseteq[r]$, the semisimplification of $\rho_s$ has bounded image.
\end{proposition}

\begin{proof}
Let $s=[x_{ij}]\in V(r)$ be such that $\rho_s(\gamma_I)$ is elliptic for every $I\subseteq[r]$. By applying Lemma \ref{genc}(3) with $I=\{i,j\}$, we see that
$$\lambda^2-(2-x_{ij}^2)\lambda+1=\det(\lambda+s_I^{-1}s_I^T)=(\lambda-\zeta_1)(\lambda-\zeta_2)$$
for some $\zeta_1,\zeta_2=\zeta_1^{-1}\in K$ with absolute value $1$ (for the unique absolute value on $K(\zeta_1)$ extending the one on $K$). If $|\cdot|$ is nonarchimedean with valuation ring $\Ocal$, then $x_{ij}^2-2=\zeta_1+\zeta_2\in\Ocal$ and hence $x_{ij}\in\Ocal$. This shows $\Img(\rho_s)\leq\GL_r(\Ocal)$ and we are done. Thus, it remains to treat the case where $|\cdot|$ is archimedean. We may assume that $K=\R$ or $\C$ and that $|\cdot|$ is the usual absolute value on $\C$. In this case, $\zeta_1=\bar\zeta_2\in\C$ and $x_{ij}^2=2-2\Re(\zeta_1)\geq0$ and hence $x_{ij}\in\R$ for each $1\leq i<j\leq r$. This shows that $s\in V(r)(\R)$, and we may henceforth assume that $K=\R$. Next, we claim that $s+s^T$ is positive semidefinite. We recall the following generalization of Sylvester's criterion:

\begin{fact}
\label{sylv}
Suppose $a$ is an $r\times r$ real symmetric matrix such that $\det(a_I)\geq0$ for every $I\subseteq[r]$. Then $a$ is positive semidefinite.
\end{fact}

It thus suffices to show that $\det((s+s^T)_I)\geq0$ for each $I\subseteq[r]$. We have
$$\det((s+s^T)_I)=\det(s_I+s_I^T)=\det(\mathbb I+s_I^{-1}s_I^T).$$
By Lemma \ref{genc}(3), the right hand side is a product of numbers of the form $1-\zeta$ where $\zeta$ is an eigenvalue of $\rho_s(\gamma_I)$. Moreover, the nonreal factors $1-\zeta$ come in conjugate pairs since $s_I$ has real entries. Since the eigenvalues of $\rho_s(\gamma_I)$ have absolute value $1$, it follows that $\det((s+s^T)_I)\geq0$. Thus $s+s^T$ is positive semidefinite.

By Lemma \ref{genc}(5), the associated representation $\tilde\rho_s$ on $W^s$ preserves the positive definite symmetric pairing induced by $s+s^T$. It follows that $\Img(\tilde\rho_s)$ lies in a compact subgroup of $\GL(W_s)$, and is in particular semisimple. The semisimplification of $\rho_s$, being the direct sum of $\tilde\rho_s$ with $\dim U_0$ copies of the trivial representation, therefore is bounded.
\end{proof}

\begin{corollary}
\label{stokesfin}
Let $k$ be a field of characteristic zero. Let $s\in V(r)(k)$ be a Stokes matrix. If $\rho_s(\gamma_I)$ has finite order for every $I\subseteq[r]$, then the semisimplification of $\rho_s$ has finite image.
\end{corollary}

\begin{proof}
First, we may assume that $k$ is algebraically closed. Let us write $s=[x_{ij}]$. By Lemma \ref{genc}(3) with $I=\{i,j\}$, we see by arguing as in the proof of Proposition \ref{elliptic} that every $x_{ij}$ is an algebraic integer, and hence the eigenvalues of $\rho_s(\gamma)$ are all algebraic integers for all $\gamma\in F_r$. Moreover, we have $\Img(\rho_s)\subseteq\GL_r(\bar\Q)$, and hence we may assume that $k=\bar\Q$. Now, for each embedding $\sigma:\bar\Q\hookrightarrow\C$, it follows by Proposition \ref{elliptic} that the image of
$$\tilde\rho_s^\sigma=\sigma\circ \tilde\rho_s:F_r\to\GL(W^s)\to\GL(W^s\otimes\C)$$
is semisimple, and in particular $\tilde\rho_s^\sigma(\gamma)$ is semisimple and has absolute $1$ for all $\gamma\in F_r$. We thus deduce that for any $\gamma\in F_r$ the eigenvalues of $\rho_s(\gamma)$ are algebraic integers all of whose conjugates have absolute value $1$, and hence are roots of unity by Kronecker's theorem. Since $\tilde\rho_s(\gamma)$ is semisimple, it follows that $\tilde\rho_s(\gamma)$ has finite order for every $\gamma\in F_r$, and therefore $\tilde\rho_s$ has finite image by Schur's theorem. Since the semisimplification of $\rho_s$ is a direct sum of $\tilde\rho_s$ and $\dim U_0$ copies of the trivial representation, the desired result follows.
\end{proof}

\section{Reflective representations}\label{sect:4}

\subsection{Reflective representations} Our aim here is to prove Theorem \ref{refthm}. Let us first introduce the notion of reflective representations. Let $(\cdot,\cdot)$ be the standard symmetric bilinear form on $\C^n$. Let $q(x_1,\dots,x_n)=x_1^2+\dots+x_n^2$ be the associated quadratic form. Let $S(n)=\{v\in \C^n:q(v)=1\}$ denote the (algebrai) unit sphere.

\begin{definition}
\label{reflective}
    Let $F_r=\langle\gamma_1,\dots,\gamma_r\rangle$ be a free group of rank $r$. Given a sequence $u=(u_1,\dots,u_r)\in S(n)^r$, the \emph{reflective representation} $\rho_u:F_r\to\GL_n(\C)$ associated to $u$ is the representation such that
    $$\rho(\gamma_i)(v)=v-2(v,u_i)u_i$$
    for all $v\in\C^n$ and $i=1,\dots,r$. We call a representation $\rho=F_r\to\GL_n(\C)$ \emph{reflective} if $\rho=\rho_u$ for some $u\in S(n)^r$.
\end{definition}

The above terminology is motivated by the fact that, if the associated vectors $u_1,\dots,u_r$ lie in $\R^n$, then the image of $F_r$ in $\Ort(n,\R)$ is a reflection group. For later use, it will be useful to set up some terminology for more general quadratic spaces. Our conventions and notations are imported from \cite{fw}. Let $k$ be a field of characteristic zero. Let $V$ be a finite-dimensional vector space over $k$, also viewed as an affine space over $k$. Let $q$ be a nondegenerate quadratic form on $V$. We define the \emph{unit sphere} $S(q)$ to be the affine hypersurface $S(q)=\{u\in V:q(u)=1\}$. The orthogonal group $\Ort(q)$ of $q$ is the automorphism group of the quadratic space $(V,q)$. In other words, $$\Ort(q)=\{g\in\GL(V):q(gv)=q(v)\text{ for all }v\in V\}\leq\GL(V).$$
The special orthogonal group of $q$ is $\SO(q)=\Ort(q)\cap\SL(V)$. Let $$\Cl(q)=\Cl^0(q)\oplus\Cl^1(q)$$ be the $\Z/2$-graded Clifford algebra associated to $(V,q)$, and let $j:V\to\Cl^1(V)$ be the canonical embedding. We denote by $\alpha$ the automorphism of $\Cl(q)$ induced by the automorphism $v\to -v$ of the quadratic space $(V,q)$.

We define the pin group $\Pin(q)$ as the subgroup of $\Cl^\times(q)$ generated by $j(S(q))$. We can write $\Pin(q)=\Pin^0(q)\sqcup\Pin^1(q)$ where $\Pin^i(q)=\Pin(q)\cap \Cl^i(q)$ for $i=0,1$. By definition, the spin group is $\Spin(q)=\Pin^0(q)$. We have a natural surjective morphism $\pi:\Pin(q)\to \Ort(q)$ which sends $g\in\Pin(q)$ to the orthogonal transformation $\pi(g)\in\Ort(q)$ of $V$ given by
$$\pi(g)(v)=\alpha(g)\otimes v\otimes g^{-1}\quad\text{for all}\quad v\in V.$$
The image of $\Spin(q)$ under $\pi$ is equal to $\SO(q)$. The following definition was introduced in \cite{fw}.
\begin{definition}
Let $(V,q)$ be a nondegenerated quadratic space over a field $k$ of characteristic zero, and let $n\geq1$ be an integer.
\begin{enumerate}
    \item The \emph{moduli space of $r$ points on $S(q)$} is
$A(r,n)=S(q)^r\git\SO(q)$.
    \item The \emph{moduli space of $r$ unoriented points on $S(q)$} is $A'(r,q)=S(q)^r\git\Ort(q)$.
\end{enumerate}
Here, $\SO(q)$ and $\Ort(q)$ act diagonally on $S(q)^r$.
\end{definition}

For $(V,q)=(\C^n,x_1^2+\dots+x_n^2)$, we recover $S(n)=S(q)$, and we shall write $\Ort(n)=\Ort(q)$, $\Cl(n)=\Cl(q)$, $A(r,n)=A(r,q)$, etc.

\begin{proposition}
    The morphism $S(n)^r\to\Hom(F_r,\GL_n)$ given by construction of the associated reflective representation $u\mapsto\rho_u$ descends to a morphism
    $$\Phi:A'(r,n)\to X(F_r,\GL_n)=\Hom(F_r,\GL_n)\git\GL_n.$$
\end{proposition}

\begin{proof}
    Since $A'(r,n)$ is the categorical quotient of $S(n)^r$ by $\Ort(n)$, it suffices to show that if $u,u'\in S(n)^r$ are $\Ort(n)$-equivalent then they have the same image in $X(F_r,\GL_n)$. But if $u=g\cdot u'$, then
    $$\rho_u(\gamma_i)(v)=v-2(v,u_i)u_i=v-2(v,g\cdot u_i')g\cdot u_i'=g(\rho_{u'}(\gamma_i)(g^{-1}v))$$
    for all $v\in V$ and $i=1,\dots,r$. This gives us the desired result.
\end{proof}

\begin{proposition}
\label{ref2}
Let $r\geq1$ be an integer.
\begin{enumerate}
    \item We have a chain of closed immersions
$$A'(r,1)\hookrightarrow A'(r,2)\hookrightarrow\dots\hookrightarrow A'(r,r)$$
induced by the standard embedding of quadratic spaces $$(\C^m,x_1^2+\dots+x_m^2)\to(\C^n,x_1^2+\dots+x_n^2)$$ for $m\leq n$.
For $r\leq m$, we have isomorphisms $A'(r,r)\simeq A'(r,m)$.
    \item We have an isomorphism $A'(r,r)\to V(r)$ given by taking $u=[u_1,\dots,u_r]$ to the unique Stokes matrix $s=s(u)$ such that
$[2(u_i,u_j)]=s+s^T.$
    \item Given integers $r\geq1$ and $1\leq m\leq n$, the diagram
    $$\xymatrix{A'(r,m) \ar[r] \ar[d]^{\Phi}& A'(r,n) \ar[d]^{\Phi}\\ X(F_r,\GL_m) \ar[r] & X(F_r,\GL_n)}$$
    commutes, where the bottom horizontal arrow is induced by the standard block diagonal embedding $\GL_m\to\GL_n$.
    \item The following diagram is commutative:
$$\xymatrix{
A'(r,r)  \ar[r]^{\sim} \ar[dr]_\Phi &V(r)\ar[d]^{s\mapsto\rho_s}\\
 &X(\Sigma_r,\GL_r)}$$
\end{enumerate}
\end{proposition}

\begin{proof}
(1) and (2) follow from the invariant theory of the orthogonal group (\cite[Chapter 11 §2.1]{procesi2} and \cite[Chapter 2 §9]{weyl}). See \cite[Proposition 5.3]{fw} for details. (3) is obvious. To prove (4), it suffices to prove the commutativity on the dense open subset of $A'(r,r)$ parametrizing those $u\in A'(r,r)$ represented by a sequence $(u_1,\dots,u_r)\in S(m)^r$ with $u_1,\dots,u_r$ linearly independent. But in this case, we see that $\rho_{s(u)}$ is conjugate to $\rho_u$ via the change-of-basis matrix $[u_1\cdots u_r]$.
\end{proof}

We now restate and prove Theorem \ref{refthm}.

\begin{thm}{\ref{refthm}}
    Let $F_r=\langle\gamma_1,\dots,\gamma_r\rangle$ be a free group of rank $r$. Then $$L_{r}=\{\gamma_{i_1}\dots \gamma_{i_u}:1\leq i_1<\dots <i_u\leq r\text{ and }1\leq u\leq r\}$$
    is a Burnside set for the class of all semisimple reflective representations of $F_r$.
\end{thm}

\begin{proof}
    Let $\rho:F_r\to\GL_n(\C)$ be a semisimple reflective representation, with associated vectors $u_1,\dots,u_r\in S(n)$. By Proposition \ref{ref2}, there exist $u_1',\dots,u_r'\in S(r)$ such that $[u_1,\dots,u_r]=[u_1',\dots,u_r']$ in $A'(r,r)$. Let $s\in V(r)(\C)$ be the unique Stokes matrix such that $s+s^T=[2(u_i',u_j')]$.
    Let $m=\max\{r,n\}$. The class of $\rho_s$ in $X(F_r,\GL_m)$ is the same as the class of $\rho$. Since $\rho$ is semisimple, it follows that $\rho$ is isomorphic to the semisimplification of $\rho_s$. Theorem \ref{refthm} then follows from Corollary \ref{stokesfin}.
\end{proof}

\subsection{Exceptional isomorphism} \label{sect:4.2}
In \cite{fw}, an exceptional isomorphism between the moduli space of points $A(r+1,4)$ defined in Section \ref{sect:4} and the characer variety $X(F_r,\SL_2)=\SL_2^r\git\SL_2$ was established. We will make use of this isomorphism in the derivation of Theorem \ref{mainthm2} from Theorem \ref{refthm}.

Let $V=\Mat_2$ be the space of $2\times 2$ matrices with quadratic from $q=\det$ given by the determinant. Let $S(q)=\{v\in V:q(v)=1\}=\SL_2$. The Clifford algebra $\Cl(q)$ associated to $(V,q)$ can be identified with the $\Z/2$-graded algebra
$$M=M^0\oplus M^1=\Mat_2^2\oplus\Mat_2^2\iota$$
where $\iota(a,b)=(b,a)\iota$ for all $(a,b)\in\Mat_2^2$. We have an embedding $j:V\to M^1$ given by $j(x)=(x,\bar x)\iota$ where $\bar x$ denotes the adjugate of $x$, namely,
$$\text{if $x=\begin{bmatrix}x_{11} & x_{12}\\ x_{21} & x_{22}\end{bmatrix}$ then $\bar x=\begin{bmatrix}x_{22} &-x_{12} \\-x_{21} & x_{11}\end{bmatrix}$}.$$
Under the above identification, the spin group $\Spin(q)=\langle j(S(q))\rangle\cap M^0$ is identified with $\SL_2^2$. The conjugation action of $\Spin(q)$ on $j(V)$ factors through the surjection $\pi:\Spin(q)\to\SO(4)$, and note that we have
$$(a,b)\cdot (u,u^{-1})\iota=(a,b)(u,u^{-1})\iota (a^{-1},b^{-1})=(aub^{-1},(aub^{-1})^{-1})\iota$$
for every $(a,b)\in\SL_2^2$ and $u\in\SL_2$. Thus, we have
$$S(q)^{r+1}\git\SO(q)=S(q)^{r+1}\git\Spin(q)=\SL_2^{r+1}\git\SL_2^2$$
where the action of $\SL_2^2$ on $\SL_2^{r+1}$ is given by
$$(a,b)\cdot(u_0,\dots,u_r)=(au_0b^{-1},\dots,au_rb^{-1})$$
for every $(a,b)\in\SL_2^2$ and $(u_0,\dots,u_r)\in\SL_2^{r+1}$. 
In light of this and the fact that the quadratic spaces $(\C^4,x_1^2+\dots +x_4^2)$ and $(V,q)$ are isomorphic over $\C$, we obtain the exceptional isomorphism
\begin{align*}
    A(r+1,4)&=S(4)^{r+1}\git\SO(4)=S(4)^{r+1}\git\Spin(4)\simeq S(q)^{r+1}\git\Spin(q)\\
    &\simeq \SL_2^{r+1}\git\SL_2^2\to\SL_2^r\git\SL_2=X(F_r,\SL_2)
\end{align*}
where the arrow $\SL_2^{r+1}\git\SL_2^2\to\SL_2^r\git\SL_2$ is given by
    $$[u_0,\dots,u_r]\mapsto[u_0 u_1^{-1},\dots,u_{r-1}u_{r}^{-1}].$$
This is the isomorphism described in \cite{fw}.

We now give an alternate description of the isomorphism $A(r+1,4)\simeq X(F_r,\SL_2)$ defined above. Let
$W_r=\langle\delta_0,\dots,\delta_r:\delta_0^2,\dots,\delta_r^2\rangle$
be the universal Coxeter group of rank $r+1$. The free group $F_r$ of rank $r$ embeds as the subgroup of $W_r$ generated by the elements $\gamma_1,\dots,\gamma_r$ where $\gamma_i=\delta_{i-1}\delta_i$ for $i=1,\dots,r$. Given $u=(u_0,\dots,u_r)\in S(q)^{r+1}$, let us consider the group homomorphism
$\rho_u:W_r\to\Pin(q)$
given by $$\rho_u(\delta_i)=j(u_i)=(u_i,u_i^{-1})\iota$$ for $i=0,\dots,r$. Note that we have
\begin{align*}
\rho(\gamma_i)&=\rho(\delta_{i-1}\delta_i)=\rho(\delta_{i-1})\rho(\delta_i)\\
&=(u_{i-1},u_{i-1}^{-1})\iota(u_{i},u_{i}^{-1})\iota=(u_{i-1},u_{i-1}^{-1})(u_i^{-1},u_i)\\
&=(u_{i-1}u_i^{-1},(u_{i-1}u_i^{-1})^{-1})
\end{align*}
for $i=1,\dots,r$. Thus, the restriction
$$\rho_u|_{F_r}:F_r\to\Spin(q)=\SL_2\times\SL_2$$
gives rise to a pair $\rho_u',\rho_u'':F_r\rightrightarrows\SL_2$ of representations by two projections to $\SL_2$, where concretely $\rho_u'(\gamma_i)=u_{i-1}u_i^{-1}$ and $\rho_u''(\gamma_i)=(u_{i-1}u_i^{-1})^{-1}$ for $i=1,\dots,r$. The isomorphism
$$A(r+1,4)\simeq A(r+1,q)\simeq X(F_r,\SL_2)$$
is induced by the assignment $u\mapsto\rho_u'$.

\begin{remark}
    We can describe the assignment in the inverse direction as follows. Given a representation $\rho':F_r\to\SL_2(\C)$, we define
    $$u=(u_0,\dots,u_r)\in S(q)^{r+1}=\SL_2^{r+1}$$
    by setting $u_0=\mathbb I$ and inductively letting    
    $ u_{i}=\rho(\gamma_i)^{-1}u_{i-1}$ for each $i=1,\dots,r$. The class of $u$ in $S(q)^{r+1}\git\SO(q)$ is well-defined and depends only on the class of $\rho'$ in $X(F_r,\SL_2)$, and provides the inverse map $X(F_r,\SL_2)\to A(r+1,q)\simeq A(r+1,4)$.
\end{remark}

\subsection{Degree $2$ representations}\label{sect:4.4}
We now discuss the derivation of Theorem \ref{mainthm2} from Theorem \ref{refthm}. First, we record a basic lemma on Burnside sets.

\begin{lemma}
\label{finlem}
    Let $F$ be a free group of finite rank. Let $G$ and $H$ be groups.
    \begin{enumerate}
        \item If $G\to H$ is a surjective morphism with finite kernel, a Burnside set for $\Hom(F,H)$ is a Burnside set for $\Hom(F,G)$ and vice versa.
        \item A Burnside set for the union $\Hom(F,G)\cup\Hom(F,H)$ is a Burnside set for $\Hom(F,G\times H)$, and vice versa.
    \end{enumerate}
\end{lemma}

\begin{proof}
    (1) A Burnside set for $\Hom(F,H)$ is a Burnside set for $\Hom(F,G)$, since a morphism $F\to G$ has finite image if and only if the composition $F\to G\to H$ has finite image. A Burnside set for $\Hom(F,G)$ is a Burnside set for $\Hom(F,G)$, since a morphism $F\to H$ lifts to a morphism $F\to G$ which has finite image if and only if the original morphism has finite image. (2) This follows from the fact that a morphism $F\to G\times H$ has finite image if and only if each of its components $F\to G$ and $F\to H$ has finite image.
\end{proof}

\begin{thm}{\ref{mainthm2}}
Let $F$ be a group generated by a finite set $S$. The set of elements of word length $\leq3|S|$ in $S$ is a Burnside set for $\Hom(F,\GL_2(\C))$.
\end{thm}

\begin{proof}
It suffices to treat the case where $F=F_r$ is a free group on $r$ generators $\gamma_1,\dots,\gamma_r$. We will explicitly construct a finite set of words in $S=\{\gamma_1,\dots,\gamma_r\}$ that provides a Burnside set for $\Hom(F_r,\GL_2(\C))$. First, note that the morphisms $$\SL_2(\C)\xrightarrow{\pr}\PSL_2(\C)\quad\text{and}\quad\GL_2(\C)\xrightarrow{(\det,\pr)}\C^\times\times\PSL_2(\C)$$ are surjective with finite fibers. In light of Lemma \ref{finlem}, we see that if $L\subseteq F_r$ is a finite Burnside set for $\Hom(F,\SL_2(\C))$ that contains $\{\gamma_1,\dots,\gamma_r\}$, then $L$ is a Burnside set for $\Hom(F,\GL_2(\C))$. It therefore suffices to construct a Burnside set for $\Hom(F,\SL_2(\C))$. We divide our analysis into the loci
$$\Hom(F,\SL_2(\C))=\Hom(F,\SL_2(\C))^{\red}\cup\Hom(F,\SL_2(\C))^{\textup{ss}}$$
of reducible representations and of semisimple representations. We first claim that the set
$$L'=\{\gamma_1,\dots,\gamma_r\}\cup\{[\gamma_i,\gamma_{j}]:1\leq i <j\leq r\}$$
is a Burnside set for $\Hom(F,\SL_2(\C))^{\red}$. Indeed, suppose that $\rho:F_r\to \SL_2(\C)$ is reducible and $\rho(\gamma)$ has finite order for every $\gamma\in L'$. By reducibility of $\rho$, we see that $\rho([F_r,F_r])$ is unipotent. It follows that $\rho([\gamma_i,\gamma_{j}])$, being of finite order, must be the identity matrix for each $1\leq i<j\leq r$. It follows that the image of $\rho$ is abelian, and hence finite since $\gamma_1,\dots,\gamma_r\in L'$. Note that each $\gamma\in L'$ has word length $\leq 4$ (and $L'=\{\gamma_1\}$ if $r=1$).

It remains to find a finite Burnside set for the locus of semisimple representations $\Hom(F_r,\SL_2(\C))^{\textup{ss}}$. Given $\rho\in\Hom(F_r,\SL_2(\C))^{\textup{ss}}$, let $u\in S(q)^{r+1}$ be a sequence of vectors such that $\rho=\rho_u'$ using the notation of Section \ref{sect:4.2}. Then we see that $\rho$ has finite image if and only if $\rho_u:W_r\to\Pin(q)$ has finite image. Since the map $\Pin(q)\to\Ort(q,\C)\simeq\Ort(4,\C)$ is surjective with finite kernel (the latter may be verified by explicit computation), we conclude by Lemma \ref{finlem} that $\rho_u$ has finite image if and only if the composition $\bar\rho_u:W_r\to \Ort_n(\C)$ has finite semisimplification. But $\bar\rho_u$ is a reflective representation, and by Theorem \ref{refthm} (or rather its proof) it has finite semisimplification if and only if $\rho_u(\gamma)$ has finite order for all $\gamma\in L'''$, where $$L'''=\{\delta_{i_1}\cdots\delta_{i_u}:0\leq i_1<\dots <i_u\leq r\text{ and }1\leq u\leq r+1\}.$$ The last condition is satisfied if and only if $\rho$ has finite order for every $\gamma\in L''$, where $L''=\{\delta^2:\delta\in L'''\}\subseteq F_r$. Finally, note that every $\gamma\in L''$ has word length at most $3r$ in $\gamma_1,\dots,\gamma_r$.
\end{proof}

\begin{corollary}
\label{surfcor}
    Let $\Sigma$ be a surface of genus $g>0$ with $n\leq2$ punctures. There is an explicit finite Burnside set $L$ consisting of simple loops on $\Sigma$ for the class of all semisimple representations $\pi_1(\Sigma)\to\GL_2(\C)$.
\end{corollary}

\begin{proof}
    We may assume without loss of generality that $\Sigma$ has at least one puncture. Let $r=2g+n-1$, so that $\pi_1(\Sigma)\simeq F_r$. Note that $\Sigma$ is a hyperelliptic double covering of the orbifold surface $\bar\Sigma$ of genus $0$ with $r+1=2g+n$ orbitfold points of index $2$ and one puncture. The orbifold fundamental group of $\bar\Sigma$ is $$W_r=\langle \delta_0,\dots,\delta_r|\delta_0^2,\dots,\delta_r^2\rangle$$
    in which $\pi_1(\Sigma)=F_r$ sits as the index $2$ subgroup generated by $\gamma_1,\dots,\gamma_r$ where we define $\gamma_i=\delta_{i-1}\delta_i$ for $i=1,\dots,r$. Since the set $L'''$ considered in the proof of Theorem \ref{refthm} above correspond to simple loops on $\bar\Sigma$, we see that the set $L''=\{\delta^2:\delta\in L'''\}\subseteq\pi_1(\Sigma)$ consists of simple loops or their squares on the double cover $\Sigma$. Since $L''$ thus constructed is a Burnside set for the semisimple representations $F_r\to\GL_2(\C)$, the claim follows.
\end{proof}

\begin{appendix}

\end{appendix}

\end{document}